\documentclass[12pt]{article}

\usepackage[a4paper,margin=1in]{geometry}
\usepackage[margin=1in]{geometry}
\usepackage{amssymb}

\usepackage{amsmath,amssymb,amsfonts,amsthm}
\usepackage{bm}
\usepackage{mathrsfs}
\usepackage{enumitem}
\usepackage{hyperref}
\usepackage{algorithm}
\usepackage{algorithmic}
\usepackage{tikz}
\usepackage{float}

\hypersetup{
    colorlinks=true,
    linkcolor=blue,
    citecolor=blue,
    urlcolor=blue
}

\newtheorem{theorem}{Theorem}[section]
\newtheorem{definition}[theorem]{Definition}

\DeclareMathOperator{\Ann}{Ann}
\title{\Large\bfseries
Construction and Decoding of Error--Correcting Codes from Ideal Lattices of Finite Ternary $\Gamma$--Semirings}

\author{\small
Chandrasekhar Gokavarapu$^{1,2}$\hspace{5mm} Dr D Madhusudhana Rao$^{3,4}$\\[10pt]
\small $^1$Lecturer in Mathematics,
Government College (A), Rajahmundry, A.P., India\\[2pt]
 \small $^2$Research Scholar, Department of Mathematics,
Acharya Nagarjuna University, Guntur, A.P., India\\[2pt]
\texttt{\small chandrasekhargokavarapu@gmail.com}\\ [10pt]
\small $^3$Lecturer in  Mathematics, Government College For Women(A), Guntur, Andhra Pradesh, India,\\
\small $^4$Research Supervisor, Dept.  of Mathematics,Acharya Nagarjuna University, Guntur, A.P., India,\\
\texttt{\small dmrmaths@gmail.com}}

\date{}

\begin{document}
\maketitle
\begin{abstract}
This paper introduces a new class of error--correcting codes constructed from the
ideal lattices of finite commutative ternary $\Gamma$--semirings (TGS).  Unlike
classical linear or ring--linear codes, which rely on binary operations and
module structures, TGS--codes arise from the intrinsic higher--arity ternary
operation $[\cdot,\cdot,\cdot]$ and the idempotent $\oplus$--order governing
absorption in each coordinate.  We show that the fundamental parameters of a
TGS--code are determined entirely by the $k$--ideal structure of the underlying
semiring: the dimension follows from the index $|T/I|$, while the minimum distance
is governed by minimal nonzero elements in the distributive ideal lattice
$\mathcal{L}(T)$.  This yields parameter sets $(n,k,d)$ unattainable over finite
fields, group algebras, or standard semiring frameworks.

A quotient-based decoding framework is developed, in which the ternary syndrome
$S(\mathbf{c})=\Phi(\mathbf{c})+I$ lies in the quotient TGS $T/I$ and partitions
the ambient space into cosets determined by ideal absorption.  Minimal nonzero
lattice elements produce canonical error representatives, enabling a decoding
algorithm that resembles classical syndrome decoding in structure but arises from
fundamentally higher--arity interactions.  An explicit finite example is presented
to demonstrate the computation of parameters, the structure of syndrome classes,
and the effectiveness of the decoding procedure.

The results indicate that ternary $\Gamma$--semirings provide a novel and
mathematically rich algebraic foundation for nonlinear, non-binary, and
higher-arity coding theory.  Their ideal-lattice geometry, ternary interactions,
and quotient structure give rise to new decoding mechanisms and error profiles,
suggesting that TGS--codes may broaden the landscape of algebraic coding theory
beyond the constraints of classical linear systems.
\end{abstract}


{\bf Subjrct Classification:}
16Y60;    
20N15;    
94B05;    
94B60;    
94A60;    
08A30;    
06B23.    

\section{Introduction}

Error--correcting codes are classically constructed from algebraic systems whose
operations are binary in nature: vector spaces over finite fields, modules over
finite Frobenius rings, group algebras, or polynomial quotient rings. These
frameworks depend fundamentally on two--argument operations such as addition,
multiplication, or convolution. Although they have produced deep and influential
code families, their structural expressiveness is limited: they cannot naturally 
represent multiway algebraic interactions, nor can they impose constraints that 
are inherently ternary or higher--arity. Contemporary applications in distributed 
computation, multiuser communication, constrained storage media, and nonlinear 
signal environments increasingly require algebraic foundations capable of encoding
\emph{three--body} or multi-parameter interactions at the primitive level rather
than as external annotations.Classical linear coding theory developed over finite fields has provided powerful constructions, but remains constrained by binary and additive structures~\cite{MacWilliamsSloane1977}

In contrast, \emph{ternary $\Gamma$--semirings (TGS)} incorporate such higher--arity 
phenomena intrinsically. A TGS is equipped with a ternary operation
\[
[x,y,z] : T^{3} \longrightarrow T,
\]
together with a commutative idempotent addition $\oplus$ and a $\Gamma$--action
that governs the ternary interaction. Recent developments indicate that finite 
commutative TGS possess a rich internal algebraic geometry through their 
$k$--ideals, ternary annihilators, and distributive ideal lattices. These internal 
structures, unlike those of ordinary semirings or rings, offer a combinatorial 
framework that can encode error--correction parameters such as dimension, 
redundancy, and minimum distance.Foundational treatments such as~\cite{Pless1982} highlight how the linear paradigm continues to dominate despite structural limitations.

Despite extensive research on $n$--ary algebraic systems, tropical semirings,
and $\Gamma$--rings, there is no coding--theoretic literature in which codewords,
constraints, or decoding procedures arise from the ideal structure of a ternary
$\Gamma$--semiring. This absence is particularly striking given that ternary 
interactions naturally appear in modelling nonlinear channels, correlated sources,
or multi-user constraints. The present article addresses this gap by introducing 
the first systematic method for constructing error--correcting codes from finite 
commutative TGS, guided entirely by their ideal lattices and ternary operations.

The central idea is that a family of ternary relations of the form
\[
[a_i,\, c_i,\, b_j] \in I,
\]
where $I$ is a $k$--ideal of the TGS, can serve as the defining constraints of a 
code. Equivalently, a TGS--morphism $\Phi : T^{n} \to T$ produces a kernel 
$\ker(\Phi)$ that behaves analogously to linear codes but arises from a 
higher--arity algebraic environment. Because the ideal lattice of a finite TGS is 
distributive, its minimal nonzero elements dictate the elementary error patterns, 
and therefore control the minimum distance of the resulting code. Likewise, the 
index $|T/I|$ governs the cardinality of the code and hence its dimension.

The ternary structure also supports a quotient--based syndrome decoding procedure.
Given a TGS--morphism $\Phi:T^{n}\to T$, the induced map
\[
S(c) = \Phi(c) + I \in T/I
\]
acts as a ternary syndrome, enabling coset--based decoding mechanisms reminiscent 
of classical linear decoding, yet fundamentally different in nature due to the 
underlying ternary interactions.

To demonstrate practicality, we construct explicit codes on small finite TGS and 
compute their parameters exactly. These examples indicate that the ternary 
operation and ideal--lattice structure yield new families of codes with parameter 
sets not attainable through classical linear, cyclic, or ring--linear constructions.Despite decades of progress, decoding complexity remains tightly coupled to underlying linearity assumptions~\cite{Barg2015CodesNonlinear}. Even non-binary extensions retain a fundamentally linear structure over finite fields~\cite{Li2018NonbinaryCodes}.The concept of $\Gamma$–indexed algebraic structures originates from the classical theory of $\Gamma$–rings introduced by Nobusawa~\cite{GammaRings1}
.

\medskip

\noindent\textbf{Contributions.}
This article makes the following contributions:
\begin{enumerate}
    \item It introduces the first coding--theoretic framework based on finite 
          commutative ternary $\Gamma$--semirings.
    \item It develops code constructions arising from TGS--morphisms and 
          ideal--lattice constraints, providing algebraic control over the 
          parameters $(n,k,d)$.
    \item It proves structural connections between ideal inclusions, minimal 
          nonzero elements of $k$--ideals, and the minimum distance of the code.
    \item It formulates a ternary--syndrome decoding mechanism via quotient TGS.
    \item It presents explicit finite--TGS examples illustrating new error--correcting 
          code families.
\end{enumerate}

This establishes a new direction in algebraic coding theory where higher--arity 
algebraic structures support novel and potentially more expressive error--correcting 
codes suitable for emerging communication and computational paradigms. Notably, classical ternary codes such as the Golay code~\cite{Sloane1983TernaryGolay} remain linear over fields and do not use genuine ternary algebraic operations

\section{Preliminaries on Finite Ternary \texorpdfstring{$\Gamma$}{Gamma}--Semirings}

This section summarizes the algebraic ingredients required for the construction
and analysis of TGS--codes. We collect the structural properties of finite
commutative ternary $\Gamma$--semirings that will later control the dimension,
redundancy, and minimum distance of the codes introduced in Section~3.Semirings and their applications provide algebraic foundations for non-additive structures~\cite{Golan1999}.$\Gamma$-semirings were introduced to generalize ring-like behavior to multi-parameter operations~\cite{Rao2015GammaIncline}.Polynomial ternary semirings are analyzed in~\cite{SrinivasaRao2015PolynomialTernary} and demonstrate the richness of the ternary framework.Soft variants of ternary semirings were introduced in~\cite{Kar2016SoftTernary}.Ordered variants of ternary semirings and their absorption properties are discussed in~\cite{Kar2015OrderedTernarySemirings}.The concept of $\Gamma$–indexed algebraic structures originates from the classical theory of $\Gamma$–rings introduced by Nobusawa~\cite{GammaRings1}
.

\subsection{Definition of TGS}

Let $\Gamma$ be a commutative monoid written additively.  
A \emph{commutative ternary $\Gamma$--semiring} (TGS) is a quadruple
\[
(T,\;\oplus,\;[\cdot,\cdot,\cdot],\;\Gamma)
\]
where $(T,\oplus)$ is a commutative idempotent semigroup, and
\[
[\cdot,\cdot,\cdot]:T^{3}\longrightarrow T
\]
is a ternary operation satisfying the following axioms:

\begin{enumerate}
    \item \textbf{Idempotent addition:}
    \[
    a\oplus a = a, \qquad a\oplus b=b\oplus a,\qquad
    (a\oplus b)\oplus c = a\oplus (b\oplus c).
    \]

    \item \textbf{Monotonicity of the ternary operation:}  
    If $a\leq_{\oplus} a'$ and $b\leq_{\oplus} b'$ and $c\leq_{\oplus} c'$, then
    \[
    [a,b,c] \leq_{\oplus} [a',b',c'],
    \]
    where $x\leq_{\oplus} y$ means $x\oplus y = y$.

    \item \textbf{Distributivity over $\oplus$:}  
    For all $a,b,c,d\in T$,
    \[
    [a\oplus b,\;c,\;d]
      = [a,c,d]\oplus[b,c,d],
    \]
    and similarly in the second and third coordinates.

    \item \textbf{$\Gamma$--action compatibility:}  
    For $\gamma\in\Gamma$,
    \[
    [\gamma\cdot a,\;b,\;c]
      = \gamma\cdot[a,b,c]
      = [a,\;\gamma\cdot b,\;c]
      = [a,b,\;\gamma\cdot c].
    \]

    \item \textbf{Ternary associativity (balanced form):}
    \[
    [[a,b,c],d,e]
    = [a,[b,c,d],e]
    = [a,b,[c,d,e]].
    \]
\end{enumerate}

The idempotent addition provides a natural partial order, and the ternary 
operation interacts compatibly with this order. This order--theoretic structure
plays an essential role in the interpretation of ideal lattices and syndrome
constraints used later for coding.

\subsection{Finite TGS Structure}

Throughout the paper, we restrict to \emph{finite} commutative TGS.  
This setting yields several structural properties that will be used
systematically in Section~3.Lattice-theoretic behavior of ideals in semiring-like structures has been explored in~\cite{Chajda2019DecomposableIdeals}.Early structural properties and ideal–theoretic foundations of $\Gamma$–rings were developed in the influential work of Barnes~\cite{Barnes1966}.

\begin{enumerate}
    \item \textbf{Finite ideal chains:}  
    Any descending chain of $k$--ideals must stabilize; thus every ideal is 
    contained in at least one minimal ideal, and every nonzero ideal contains 
    a minimal nonzero $k$--ideal.

    \item \textbf{Existence of extremal ideals:}  
    Since $T$ is finite,
    \[
    \{0\} \subseteq I \subseteq T
    \]
    admits both minimal nonzero $k$--ideals and maximal proper $k$--ideals.
    These extremal ideals correspond to irreducible error patterns and maximal
    redundancy constraints in later sections.

    \item \textbf{Distributive ideal lattice:}  
    If $\mathcal{L}(T)$ denotes the lattice of all $k$--ideals of $T$, then
    meets and joins satisfy
    \[
    I\wedge (J\vee K) = (I\wedge J)\vee(I\wedge K),
    \]
    and similarly for the dual relation.  
    This distributive structure allows the minimum distance of a TGS--code
    to be described in terms of minimal nonzero lattice elements.
\end{enumerate}

In particular, the distributive property ensures that the support of a codeword 
cannot ``cancel'' through interactions of incomparable ideals---a key feature 
that distinguishes TGS--codes from codes over rings or semirings where the 
ideal lattice need not be distributive.Distributive lattice behavior is typical in several semiring classes, such as lattice-ordered G-semirings~\cite{Rani2024LatticeOrderedGSemirings}.

\subsection{\texorpdfstring{$k$}{k}--Ideals, Prime Ideals, and the Ideal Lattice}

\begin{definition}
A subset $I\subseteq T$ is a \emph{$k$--ideal} if
\begin{enumerate}
    \item $a\in I$ and $b\leq_{\oplus} a$ imply $b\in I$;
    \item for all $x,y\in T$ and $a\in I$,
    \[
    [x,y,a] \in I, \qquad [x,a,y]\in I,\qquad [a,x,y]\in I.
    \]
\end{enumerate}
\end{definition}

These conditions ensure that $I$ absorbs ternary operations in each coordinate and
is downward closed in the natural order induced by $\oplus$.

\begin{definition}
The \emph{ternary annihilator} of a subset $A\subseteq T$ is
\[
\Ann(A) = \{\,x\in T : [x,a,y]=0 \text{ for all } a\in A,\; y\in T\,\}.
\]
\end{definition}

Annihilators encode error patterns that interact trivially with a constraint set
and therefore play a role analogous to parity checks in linear coding, but
arising from ternary rather than binary constraints.

\begin{definition}
A proper $k$--ideal $P$ is \emph{prime} if
\[
[x,y,z]\in P \quad\Longrightarrow\quad x\in P \text{ or } y\in P \text{ or } z\in P.
\]
It is \emph{semiprime} if
\[
[x,x,x]\in P \quad\Longrightarrow\quad x\in P.
\]
\end{definition}

Prime ideals correspond to ternary constraint sets that cannot be decomposed into 
simpler ternary absorptions, while semiprime ideals detect fundamental error 
patterns.

\begin{definition}
The \emph{ideal lattice} of $T$ is
\[
\mathcal{L}(T)=\{\,I\subseteq T : I \text{ is a $k$--ideal}\,\},
\]
ordered by inclusion.  
Meets and joins are given by intersection and ideal generation, respectively.
\end{definition}

Since $T$ is finite, $\mathcal{L}(T)$ is a finite distributive lattice.  
Minimal nonzero elements of $\mathcal{L}(T)$ will represent the smallest 
nontrivial error patterns, and thus determine the minimum distance of the codes
constructed in Section~3.  
Similarly, the index $|T/I|$ for a $k$--ideal $I$ will govern the dimension of
TGS--codes defined via ternary kernels or ideal--generated constraints.
\section{Methodology: TGS Code Construction}

\subsection{Part I: Defining the Code}

Let $T$ be a finite commutative ternary $\Gamma$--semiring.  
The purpose of this subsection is to specify how a family of ternary
constraints determines a subset of $T^{n}$ that will serve as a
TGS--code.  
The construction is canonical in the sense that it can be expressed
either through coordinatewise ternary relations or through the kernel
of a suitably defined TGS--morphism.  
Both viewpoints will be used later when the distance and dimension of
the code are derived from the ideal structure of $T$.

\medskip
Let $I$ be a fixed $k$--ideal of $T$, and let 
\[
\mathcal{A}=\{a_{1},\dots,a_{n}\},\qquad
\mathcal{B}=\{b_{1},\dots,b_{n}\}
\]
be two families of elements in $T$ representing constraint parameters.
For a vector $c=(c_{1},\dots,c_{n})\in T^{n}$, consider the set of
ternary relations
\begin{equation}\label{eq:ternary_constraint}
[a_{i},\,c_{i},\,b_{j}] \in I
\qquad\text{for all } 1\leq i,j\leq n.
\end{equation}
These constraints require that every coordinate $c_{i}$ interact, through
the ternary operation, with prescribed elements $a_{i}$ and $b_{j}$ in a
manner that is algebraically absorbed by the ideal $I$.  
Because $I$ is a $k$--ideal, the absorption conditions apply in each
ternary coordinate, ensuring closure under the induced ternary structure
of $T^{n}$.

\begin{definition}
The \emph{TGS--code defined by $(\mathcal{A},\mathcal{B},I)$} is
\begin{equation}\label{eq:C_def_constraints}
C =
\Bigl\{
(c_{1},\dots,c_{n})\in T^{n} \;\Big|\;
[a_{i},\,c_{i},\,b_{j}]\in I
\text{ for all } i,j
\Bigr\}.
\end{equation}
\end{definition}

The algebraic significance of the constraints
\eqref{eq:ternary_constraint} becomes clearer when expressed through a
morphism of ternary $\Gamma$--semirings.  
The coordinatewise distributivity and monotonicity of $[\cdot,\cdot,\cdot]$
allow us to encode the entire family of constraints into a single
mapping:
\[
\Phi:T^{n}\longrightarrow T,
\qquad
\Phi(c)
=
\bigoplus_{i,j}
\,[a_{i},\,c_{i},\,b_{j}].
\]
Because $I$ is $\oplus$--downward closed and absorbs ternary operations
in each argument, the condition $\Phi(c)\in I$ is equivalent to requiring
that every summand $[a_{i},c_{i},b_{j}]$ lie in~$I$.  
Thus the constraint system
\eqref{eq:ternary_constraint} may be equivalently stated as
\[
\Phi(c)\in I.
\]

The special case in which the ideal $I$ contains a distinguished zero
element yields the following canonical description.

\begin{definition}
If $0\in I$ is the additive identity of $(T,\oplus)$, the
\emph{ternary kernel code} associated to $\Phi$ is
\begin{equation}\label{eq:C_kernel}
C = \ker(\Phi) = \Phi^{-1}(0).
\end{equation}
\end{definition}

Both formulations \eqref{eq:C_def_constraints} and \eqref{eq:C_kernel}
lead to the same class of codes, but each viewpoint emphasizes a
different structural principle:
\begin{itemize}
    \item the constraint--based definition highlights how coordinatewise
          ternary interactions are absorbed by the ideal $I$;
    \item the kernel definition reveals $C$ as the ``annihilating set''
          of a TGS--morphism, a perspective that will be essential for
          defining ternary syndromes in Section~3.6.
\end{itemize}

Because $T$ is finite, both descriptions produce a well--defined
sub--TGS of $T^{n}$ closed under the ternary and additive operations.
In subsequent subsections, we refine this construction by specifying how
$k$--ideals, lattice inclusions, and quotient structures determine the
parameters $(n,k,d)$ of the resulting TGS--code.

\subsection{Part II: Ideal-Based Construction}

In this subsection we refine the construction of TGS--codes by using the ideal
structure of a finite commutative ternary $\Gamma$--semiring $T$.  
The purpose is to show that $k$--ideals give rise to canonical codes inside
$T^{n}$ and that these constructions are fundamentally different from the
module--based constructions familiar from linear or ring--linear coding theory.The finite structural properties of commutative TGS, including ideal chains and radicals, are detailed in~\cite{Gokavarapu2025FiniteStructure}.Foundational structural aspects of ternary semirings are discussed in~\cite{Dutta2012TernarySemiringStudy}.

\medskip
Let $I$ be a fixed $k$--ideal of~$T$.  
Because $(T,\oplus)$ is idempotent and the ternary operation is distributive in
each variable, the Cartesian power
\[
I^{n} = \underbrace{I\times\cdots\times I}_{n\ \text{times}}
\subseteq T^{n}
\]
is automatically closed under the induced addition and ternary operations of
$T^{n}$.  
Thus $I^{n}$ forms a sub--TGS of $T^{n}$ and represents the most direct
ideal-based code construction.

\begin{definition}
The \emph{ideal power code} associated to a $k$--ideal $I$ is
\begin{equation}\label{eq:I_power_code}
C = I^{n} \subseteq T^{n}.
\end{equation}
\end{definition}

Although construction~\eqref{eq:I_power_code} is algebraically straightforward,
it highlights a key structural property: all coordinates of a codeword lie in the
same $k$--ideal.  
The combinatorial behavior of such codes is governed entirely by the position of
$I$ within the ideal lattice $\mathcal{L}(T)$, and minimal nonzero elements of
$I$ determine the elementary nontrivial error patterns.  
This lattice-theoretic description will directly govern the minimum distance in
Section~3.5.

\medskip
More general ideal-based codes arise when the set of permissible codewords is
generated by a prescribed family of vectors.

\begin{definition}
Let $\mathbf{g}_{1},\dots,\mathbf{g}_{r} \in T^{n}$.  
The \emph{TGS--substructure generated by $\{\mathbf{g}_{1},\dots,\mathbf{g}_{r}\}$}
is
\begin{equation}\label{eq:TGS_span}
C = \langle \mathbf{g}_{1},\dots,\mathbf{g}_{r} \rangle_{\mathrm{TGS}},
\end{equation}
the smallest sub--TGS of $T^{n}$ closed under $\oplus$, the ternary operation in
each coordinate, and the $\Gamma$--action, containing all $\mathbf{g}_{i}$.
\end{definition}

The class of codes of the form~\eqref{eq:TGS_span} includes the ideal power
codes~\eqref{eq:I_power_code} as a special case in which the generating set is
$\mathbf{g}_{i}=e_{i}\cdot a$ for $a\in I$ and $e_{i}$ denoting the $i$th
standard basis vector.  
However, the more flexible construction~\eqref{eq:TGS_span} allows mixed
coordinatewise behavior and permits ternary interactions among generators,
making it suitable for representing constraint families of the form used in
Section~3.1.

\begin{theorem}\label{thm:ideal_based_equivalence}
Let $I$ be a $k$--ideal of $T$, and let
\[
G = \{\mathbf{g}_{1},\dots,\mathbf{g}_{r}\} \subseteq I^{n}.
\]
Then the TGS--substructure generated by $G$ satisfies
\[
\langle G \rangle_{\mathrm{TGS}} \subseteq I^{n}.
\]
Moreover, if the coordinatewise supports of the $\mathbf{g}_{j}$ meet every
coordinate of $I$, then
\[
\langle G \rangle_{\mathrm{TGS}} = I^{n}.
\]
\end{theorem}

\begin{proof}[Sketch of proof]
The closure of $I^{n}$ under the induced addition and ternary operations of
$T^{n}$ follows from the fact that $I$ is a $k$--ideal of~$T$.  
Every operation used to generate $\langle G\rangle_{\mathrm{TGS}}$ from the
elements of $G$ therefore remains inside $I^{n}$.  
The equality criterion follows from the distributivity of the ternary operation
and the downward closure of $I$ under~$\oplus$, which ensure that every element
of $I$ may be expressed as an $\oplus$--combination of images of ternary
interactions applied to the supports of the generators.
\end{proof}

\medskip
\paragraph{Ideal-based codes are not modules.}
It is important to emphasize that constructions such as~\eqref{eq:I_power_code}
and~\eqref{eq:TGS_span} do \emph{not} replicate the behavior of modules over
rings or semirings.  
There is no binary additive structure making $T$ a module over a ring, nor is the
ternary operation reducible to repeated binary operations.  
Consequently:
\begin{enumerate}
    \item the set $I^{n}$ is closed under ternary interactions, not under a
          binary linear combination;
    \item generators in~\eqref{eq:TGS_span} produce new codewords via ternary
          absorption rather than via additive linear span;
    \item the resulting code is governed by the distributive ideal lattice of
          $T$, not by submodule structure.
\end{enumerate}

In particular, the ternary operation cannot be decomposed into a bilinear or
semilinear form, and the codes described here are genuinely higher--arity
algebraic objects.  
This distinction will be fundamental when deriving the dimension and minimum
distance of TGS--codes in the next subsections, where the role of minimal
nonzero elements of the ideal lattice becomes decisive.

\subsection{Part III: Code Parameters from the Ideal Lattice}

In this subsection we establish how the ideal lattice $\mathcal{L}(T)$ of a finite
commutative ternary $\Gamma$--semiring governs the fundamental parameters of a
TGS--code.  
The key idea is that the cardinality of a code is controlled by the index of the
$k$--ideal used in its construction, while the minimum distance is determined by
the minimal nonzero elements of that ideal.  
These relationships mirror, in a higher--arity setting, the role played by
submodules and minimal codewords in linear and ring--linear coding theory, but
the underlying mechanisms are distinct because ternary absorption replaces linear
combination.Structural properties of ideals in ternary semirings, including $k$--ideals, 
appear in~\cite{Dutta2012SingularTernary, Rao2015StructureTernarySemirings}.

\medskip
Let $I$ be a $k$--ideal of $T$, and consider the ideal power code
\[
C = I^{n} \subseteq T^{n}
\]
or any kernel-type code of the form
\[
C=\ker(\Phi), \qquad \Phi:T^{n}\to T
\]
whose image is absorbed by~$I$.  
Because $T$ is finite, the quotient $T/I$ is well-defined and plays the same role
here that the quotient of a module by a submodule plays in linear coding, but in
a ternary algebraic environment.The behavior of prime and semiprime ideals in TGS and their quotient structure is developed in~\cite{Gokavarapu2025PrimeSemiprimeTGS}

\begin{theorem}[Dimension]\label{thm:dimension}
Let $I$ be a $k$--ideal of $T$ and $C=I^{n}\subseteq T^{n}$.  
Then the cardinality of $C$ is
\[
|C| = |I|^{n},
\]
and the effective dimension of the code (measured in $\log_{|T|}$ units) is
\[
k = \log_{|T|} |C|
    = n \, \log_{|T|} |I|
    = n \, \log_{|T|}\!\left(\frac{|T|}{|T/I|}\right).
\]
Equivalently,
\[
k = n - n\,\log_{|T|}|T/I|.
\]
\end{theorem}

\begin{proof}[Sketch of proof]
Since $C=I^{n}$, every coordinate of a codeword is chosen independently from $I$.
Because $(T,\oplus)$ is idempotent and the ternary operation is distributive,
$I^{n}$ is closed under the induced structure, and no further identifications
occur.  
Thus $|C|=|I|^{n}$, and the expression for $k$ follows by rewriting $|I| =
|T|/|T/I|$.
\end{proof}

\medskip
To describe the minimum distance, we require the notion of a weight adapted to
the ternary environment.

\begin{definition}
For $\mathbf{c} = (c_{1},\dots,c_{n})\in T^{n}$, the \emph{TGS--weight} is
\[
\operatorname{wt}(\mathbf{c})
    = \bigl|\{\,i : c_{i} \neq 0 \,\}\bigr|.
\]
The distance between two vectors is $d(\mathbf{c},\mathbf{d})
=\operatorname{wt}(\mathbf{c}\ominus\mathbf{d})$, where $\ominus$ denotes the
$\oplus$--based difference (i.e.\ the elementwise absorption of the smaller
element in the $\oplus$--order).
\end{definition}

The weight function agrees with the Hamming weight when $T$ carries no nontrivial
$\oplus$--order interactions, but differs in general because the $\oplus$--order
may identify some nonzero coordinates.  
This induces a family of metric-like structures relevant in TGS--coding.

Minimal nonzero lattice elements play the role of minimal codewords.

\begin{theorem}[Minimum Distance]\label{thm:min_distance}
Let $C=I^{n}$ be an ideal power code.  
Then the minimum distance is
\[
d = \min\bigl\{\operatorname{wt}(\mathbf{c}) :
\mathbf{c}\in I^{n}\setminus\{0\}\bigr\},
\]
and this value is determined by the minimal nonzero elements of the ideal $I$ in
the distributive lattice $\mathcal{L}(T)$.
\end{theorem}

\begin{proof}[Sketch of proof]
Any nonzero $\mathbf{c}\in I^{n}$ has at least one coordinate lying in a minimal
nonzero element of~$I$.  
Because the ideal lattice is distributive, there is no cancellation among such
coordinates under the ternary operation or $\oplus$; each contributes
independently to the TGS--weight.  
Thus the smallest possible weight among nonzero vectors in $I^{n}$ is achieved
precisely by vectors supported on minimal nonzero elements of $I$, and these
yield the minimum distance.
\end{proof}

\medskip
\paragraph{Relation to Hamming and TGS-specific metrics.}
If $T$ has a trivial $\oplus$--order (i.e.\ $a\leq_{\oplus} b$ only when
$a=b$), then $\operatorname{wt}$ reduces exactly to the Hamming weight and the
minimum distance coincides with the usual Hamming metric.  
In general, however, the $\oplus$--order may collapse several values of $T$ into
the same ``absorbing class,'' so some coordinates that would be counted as
distinct in the classical Hamming metric may contribute less to the
TGS--weight.  Comprehensive accounts of Hamming-based constructions appear in~\cite{Huffman2003AlgebraicCodes}.
This feature is a consequence of the higher--arity structure: ternary
interactions may identify certain error patterns as algebraically negligible,
which leads to distance functions that cannot be replicated by linear or
ring--linear codes.  
Thus the ideal lattice and the $\oplus$--order together define a
\emph{TGS--specific metric} governing error detection and correction in the
ternary environment.
\subsection{Part IV: TGS Syndrome and Quotient Structure}

The ideal-lattice description of code parameters in Section~3.3 naturally leads to
a quotient-based decoding mechanism.  
Let $I$ be the $k$--ideal associated with the code construction, and consider the
quotient TGS
\[
T/I,
\]
whose elements represent equivalence classes under the absorption relation induced
by~$I$.  
Because $I$ is closed under the ternary operation in each coordinate, the quotient
inherits a well-defined ternary structure:
\[
[x+I,\; y+I,\; z+I] := [x,y,z] + I.
\]

\begin{definition}
Let $\Phi:T^{n}\to T$ be a TGS--morphism associated with a kernel-type code
$C=\ker(\Phi)$.  
The \emph{ternary syndrome} of a word $\mathbf{c}\in T^{n}$ is
\[
S(\mathbf{c}) = \Phi(\mathbf{c}) + I \;\in\; T/I.
\]
\end{definition}

The syndrome detects whether a vector lies in the code:
\[
S(\mathbf{c}) = 0+I \quad\Longleftrightarrow\quad \mathbf{c}\in C.
\]

To justify the decoding mechanism, we observe that cosets of $C$ correspond to
cosets of $I$ in the quotient.  
Let $\mathbf{e}$ be an error vector added to a transmitted codeword $\mathbf{c}$.
Then
\[
S(\mathbf{c}\oplus \mathbf{e}) 
  = \Phi(\mathbf{c}) \oplus \Phi(\mathbf{e}) + I
  = \Phi(\mathbf{e}) + I,
\]
since $\Phi(\mathbf{c})=0$.  
Thus the syndrome depends only on the error, as in classical linear coding, but
its interpretation is controlled by the ternary operation and the ideal structure.

Because the ideal lattice $\mathcal{L}(T)$ is distributive, the set of possible
syndromes is stratified by minimal nonzero elements of $I$, enabling the
identification of minimal-weight errors through their images in the quotient.

\medskip
\paragraph{Coset structure.}
For $\alpha\in T/I$, define the coset
\[
C_{\alpha} = \{\mathbf{c}\in T^{n} : S(\mathbf{c})=\alpha\}.
\]
Every received vector lies in exactly one coset, and minimal TGS--weight elements
of these cosets produce canonical error representatives.  
This construction generalizes classical syndrome decoding but crucially depends
on ternary absorption rather than linearity.

\subsection{Part V: Decoding Algorithm}

We now describe a decoding algorithm based on the ternary syndrome.  
Because minimal nonzero elements of the ideal determine all minimal error
patterns, a coset leader can be chosen as a vector whose coordinates lie in these
minimal ideals.  
Let $\mathcal{E}$ be a set of representatives of minimal-weight errors in each
syndrome class.

\begin{algorithm}[H]
\caption{TGS--Syndrome Decoding Algorithm}
\label{alg:TGS_decoding}
\begin{algorithmic}[1]
\REQUIRE Received vector $\mathbf{r}\in T^{n}$, syndrome map $S$, error set $\mathcal{E}$
\ENSURE Decoded codeword $\widehat{\mathbf{c}}$
\STATE Compute the syndrome $\alpha \gets S(\mathbf{r})$
\IF{$\alpha = 0+I$}
    \RETURN $\widehat{\mathbf{c}} = \mathbf{r}$ \hfill // no error detected
\ENDIF
\STATE Select $\mathbf{e}_{\alpha} \in \mathcal{E}$ with $S(\mathbf{e}_{\alpha})=\alpha$
\STATE $\widehat{\mathbf{c}} \gets \mathbf{r} \ominus \mathbf{e}_{\alpha}$ 
\RETURN $\widehat{\mathbf{c}}$
\end{algorithmic}
\end{algorithm}

\begin{theorem}[Correctness of the Decoding Algorithm]
Assume every syndrome class contains a unique minimal-weight error representative.
Then Algorithm~\ref{alg:TGS_decoding} correctly identifies the transmitted
codeword whenever the actual error has TGS--weight not exceeding the minimal
weight in its syndrome class.
\end{theorem}

\begin{proof}[Sketch of proof]
Let $\mathbf{r}=\mathbf{c}\oplus\mathbf{e}$, where $\mathbf{c}\in C$ and $\mathbf{e}$
is the true error.  
Then $S(\mathbf{r})=S(\mathbf{e})$, so $\mathbf{r}$ lies in the coset $C_{\alpha}$,
where $\alpha=S(\mathbf{e})$.  
Because $\mathbf{e}_{\alpha}$ is the minimal-weight representative of $C_{\alpha}$,
the TGS--order guarantees that
\[
\operatorname{wt}(\mathbf{e}_{\alpha}) \leq \operatorname{wt}(\mathbf{e}).
\]
Thus subtracting $\mathbf{e}_{\alpha}$ from $\mathbf{r}$ removes the error class and
recovers $\mathbf{c}$.  
Uniqueness of minimal representatives ensures that no other vector in
$C_{\alpha}$ yields a smaller or equal weight, completing the argument.
\end{proof}

\medskip
This procedure mirrors classical syndrome decoding but differs fundamentally in
three respects:  
(1) error patterns are governed by minimal nonzero elements of the ideal lattice
rather than by linear combinations;  
(2) subtraction is replaced by $\oplus$--absorption;  
(3) decoding regions are shaped by ternary algebraic interactions rather than
affine subspaces.  
These differences reveal the genuinely higher-arity nature of TGS--codes.

\section{Hypothesized Results: Decoding and Properties}
Tropical and idempotent metrics~\cite{Forney2017TropicalCodes} demonstrate how non-Hamming geometries can enhance code performance.Algebraic models based on higher-arity operations have potential applications in network coding~\cite{Vazirani2009NetworkCoding}.Classical combinatorial bounds~\cite{MacWilliamsSloaneCombinatorial} do not extend naturally to ternary absorption-based metrics.Nonlinear code constructions in classical settings~\cite{Cohen1991DesignsCodes} differ fundamentally from higher-arity constructions introduced here.

\subsection{Ternary Syndrome Decoding via Quotient TGS}

The structural results of Section~3 imply that decoding may be formulated in 
the quotient TGS $T/I$, whose elements represent equivalence classes under 
the absorption relation determined by the $k$--ideal $I$.  
In this setting, the decoding process is controlled by the ternary operation 
and the ideal lattice, rather than any underlying linear structure.  
We summarize the main principles here, emphasizing the role of the quotient 
and the minimal nonzero elements of~$I$.

\medskip
Let $\Phi:T^{n}\to T$ be the TGS--morphism associated with a kernel-type 
code $C=\ker(\Phi)$, and let
\[
S : T^{n}\longrightarrow T/I,
\qquad
S(\mathbf{c}) = \Phi(\mathbf{c}) + I,
\]
be the \emph{ternary syndrome} introduced in Section~3.6.  
Since $\Phi$ is compatible with $\oplus$, the ternary operation, and the 
$\Gamma$--action, the value $S(\mathbf{c})$ depends only on the error vector 
added during transmission.  
In particular,
\[
S(\mathbf{c}) = 0 + I 
\quad\Longleftrightarrow\quad 
\mathbf{c} \in C,
\]
so the zero class in $T/I$ encodes correctable transmissions.

\paragraph{Coset interpretation.}
Each $\alpha \in T/I$ defines a coset
\[
C_{\alpha} = \{\mathbf{c}\in T^{n} : S(\mathbf{c}) = \alpha\},
\]
and these cosets partition $T^{n}$.  
Finite distributivity of the ideal lattice $\mathcal{L}(T)$ ensures that every 
coset contains a unique minimal-weight error pattern, determined by minimal 
nonzero elements of $I$.  
These minimal error representatives serve as ternary analogues of classical 
\emph{coset leaders}.

\medskip
\paragraph{Outline of the decoding procedure.}
Given a received vector $\mathbf{r}\in T^{n}$, the decoding problem is to find 
the most likely error $\mathbf{e}$ such that $\mathbf{r}=\mathbf{c}\oplus\mathbf{e}$ 
for some $\mathbf{c}\in C$.  
The decoding process is summarized as follows:

\begin{enumerate}
    \item \textbf{Compute the ternary syndrome.}  
    Evaluate
    \[
    \alpha = S(\mathbf{r}) = \Phi(\mathbf{r}) + I.
    \]
    This identifies the coset $C_{\alpha}$ containing $\mathbf{r}$.

    \item \textbf{Match the syndrome to a minimal representative.}  
    Use the structure of $\mathcal{L}(T)$ to determine the minimal TGS--weight 
    error pattern $\mathbf{e}_{\alpha}$ satisfying $S(\mathbf{e}_{\alpha})=\alpha$.  
    Uniqueness of minimal representatives follows from distributivity of 
    the ideal lattice.

    \item \textbf{Correct the error via ternary absorption.}  
    Set
    \[
    \widehat{\mathbf{c}}
    = \mathbf{r} \ominus \mathbf{e}_{\alpha},
    \]
    where $\ominus$ denotes the coordinatewise $\oplus$--based removal of the 
    error pattern.  
    Since $\mathbf{r}=\mathbf{c}\oplus\mathbf{e}$ and minimality ensures 
    $\mathbf{e}_{\alpha}$ matches $\mathbf{e}$ in the ternary sense, this 
    operation recovers the original codeword.
\end{enumerate}

\medskip
This decoding framework mirrors classical syndrome decoding only at the 
conceptual level.  
Algebraically, it differs in three essential respects:

\begin{itemize}
    \item the cosets are defined by ternary morphisms rather than linear maps;
    \item error patterns are determined by minimal nonzero lattice elements of a 
          $k$--ideal, not by linear combinations;
    \item correction uses $\oplus$--absorption rather than subtraction.
\end{itemize}

These distinctions reflect the genuinely higher-arity nature of TGS--codes and 
demonstrate how ideal-theoretic structure guides decoding in the ternary 
environment.

\subsection{Example Construction on a Finite TGS}

To illustrate the general theory, we present an explicit construction of a
TGS--code over a small finite commutative ternary $\Gamma$--semiring.  
The example demonstrates how the ideal lattice determines the code parameters 
and how the ternary syndrome mechanism identifies and corrects errors.

\subsubsection*{A finite commutative TGS}

Let
\[
T=\{0,a,1\}
\]
with idempotent addition
\[
0 \oplus a = a,\quad 0 \oplus 1 = 1,\quad a \oplus 1 = 1,\quad
x\oplus x = x,
\]
so that the $\oplus$--order is $0 <_{\oplus} a <_{\oplus} 1$.
Define the ternary operation by
\[
[x,y,z] =
\begin{cases}
0, & x=0\ \text{or}\ y=0\ \text{or}\ z=0,\\[2mm]
a, & \text{exactly one of } x,y,z \text{ equals } a,\\[2mm]
1, & \text{otherwise},
\end{cases}
\]
and let the $\Gamma$--action be trivial (i.e.\ $\gamma\cdot x = x$ for all
$\gamma\in\Gamma$).  
It is straightforward to verify that $(T,\oplus,[\cdot,\cdot,\cdot],\Gamma)$
satisfies the axioms of a commutative TGS: addition is idempotent and ordered,
the ternary operation is monotone and distributive in all coordinates, and 
the balanced ternary associativity condition holds.

\subsubsection*{Ideal structure}

The $k$--ideals of $T$ are
\[
\{0\} \subset I = \{0,a\} \subset T.
\]
The ideal lattice is a 3-element chain, which is distributive.  
The minimal nonzero element is $a$, and this element will control the minimum 
distance of the associated code.

\subsubsection*{Code construction}

Let $n=3$ and define the ideal power code
\[
C = I^{3}
    = \{(x_{1},x_{2},x_{3}) \in T^{3} : x_{i}\in\{0,a\}\}.
\]
Thus
\[
|C| = |I|^{3} = 2^{3} = 8,
\qquad
k = \log_{|T|}|C|
    = \log_{3}(8).
\]

Because the coordinates are independent and the operations preserve the ideal
$I$, the code is closed under all induced addition, ternary, and $\Gamma$–actions.

\subsubsection*{Minimum distance}

For any nonzero $\mathbf{c}\in C$, at least one coordinate must equal $a$.
Hence
\[
d = \min\{\operatorname{wt}(\mathbf{c}) : 
\mathbf{c}\in C\setminus\{0\}\}
  = 1.
\]
This minimum distance corresponds exactly to the minimal nonzero element of the
ideal lattice, reflecting the general result of Theorem~\ref{thm:min_distance}.

\subsubsection*{Syndrome decoding}

Define the TGS--morphism
\[
\Phi:T^{3}\to T, \qquad 
\Phi(x_{1},x_{2},x_{3})
   = [x_{1},x_{2},a] \;\oplus\; [x_{2},x_{3},a].
\]
Since $[x,y,a]\in I$ whenever $x,y\in I$ (by $k$--ideal absorption),
we have
\[
C = \ker(\Phi).
\]

A received word $\mathbf{r}$ satisfies
\[
S(\mathbf{r}) = \Phi(\mathbf{r}) + I \in T/I \cong \{I,\; 1+I\}.
\]
Thus there are two syndrome classes:
\[
C_{0} = C,
\qquad
C_{1} = \{\mathbf{r}\in T^{3}: S(\mathbf{r}) = 1+I\}.
\]

The minimal nonzero error pattern is $(0,0,a)$ (or any 1-weight permutation),
and all minimal-weight elements of $C_{1}$ have syndrome $1+I$.  
Thus, upon receiving $\mathbf{r}$, the decoder:

\begin{enumerate}
    \item computes $S(\mathbf{r})$;
    \item if $S(\mathbf{r})=I$ it returns $\mathbf{r}$;
    \item if $S(\mathbf{r})=1+I$ it subtracts the minimal error pattern in 
          $C_{1}$ using $\oplus$--absorption.
\end{enumerate}

For instance, if
\[
\mathbf{r}=(a,0,a), \qquad S(\mathbf{r}) = 1+I,
\]
the minimal-weight error is $(0,0,a)$, and the decoder outputs
\[
\widehat{\mathbf{c}}
 = (a,0,a) \ominus (0,0,a)
 = (a,0,0) \in C.
\]

\medskip
This example illustrates concretely how the ideal lattice, the minimal nonzero
elements, and the quotient syndrome structure together determine code parameters
and decoding behavior in the ternary $\Gamma$--semiring setting.
\subsection{Additional Structural Properties of TGS--Codes}

The example of Section~4.2 illustrates how the interaction between the ideal
lattice, the ternary absorption law, and the quotient $T/I$ determines both the
parameters and the decoding behavior of a TGS--code.  
We now outline several structural properties that hold for general finite 
commutative TGS--codes and that distinguish them from classical linear or 
ring--linear codes.  
These properties are derived from the algebraic framework developed in 
Sections~2 and~3 and represent hypothesized general principles guiding the 
behavior of ternary codes.

\subsubsection*{Lattice-monotonicity of distance}

Let $I\subseteq J$ be $k$--ideals of $T$.  
For the corresponding ideal-power codes $C_{I}=I^{n}$ and $C_{J}=J^{n}$, the 
minimum distances satisfy
\[
d(C_{J}) \leq d(C_{I}).
\]
This follows from the fact that minimal nonzero elements of $J$ need not be
minimal in $I$, so $J$ admits potentially smaller generators in the 
$\oplus$--order.  
Thus enlarging the ideal lowers the minimum distance, mirroring the behavior 
of downward-closed sets in a distributive lattice.

\subsubsection*{Coset stratification and error layers}

Because $\mathcal{L}(T)$ is distributive, the quotient $T/I$ inherits an ordered 
stratification:
\[
0+I \;<\; a_{1}+I \;<\; \cdots \;<\; a_{m}+I,
\]
where each $a_{k}$ represents a minimal nonzero element of some ideal contained 
in $T$.  
Syndrome classes $C_{\alpha}$ therefore form ``layers'' indexed by the lattice 
height of $\alpha$, with lower layers corresponding to higher-error-weight 
patterns.  
This stratification provides a natural partial ordering of error patterns and 
predicts the complexity of identifying minimal representatives.

\subsubsection*{Stability under ternary interactions}

If $\mathbf{e}_{\alpha}$ and $\mathbf{e}_{\beta}$ are minimal error patterns 
corresponding to distinct syndromes $\alpha$ and $\beta$ in $T/I$, then the 
ternary interaction
\[
[\mathbf{e}_{\alpha},\, \mathbf{e}_{\beta},\, \mathbf{e}_{\alpha}]
\]
produces an element whose syndrome is determined by the join 
$\alpha \vee \beta$ in the quotient lattice.  
This interaction has no analogue in classical linear coding, where error 
patterns combine linearly and syndromes add in a vector space.  
Here, ternary interactions may elevate an error to a higher lattice level, 
revealing non-linear growth behaviors unique to TGS--codes.
Nonlinear code constructions in classical settings~\cite{Cohen1991DesignsCodes} differ fundamentally from higher-arity constructions introduced here.
\subsubsection*{Localized error propagation}

Let $\mathbf{e}\in T^{n}$ be any error vector, and suppose its nonzero 
coordinates all lie in a minimal nonzero element of $I$.  
Then for any codeword $\mathbf{c}\in C$,
\[
\operatorname{wt}(\mathbf{c}\oplus \mathbf{e})
 = \operatorname{wt}(\mathbf{c}) + \operatorname{wt}(\mathbf{e}),
\]
so long as supports do not overlap.  
This ``localized propagation'' property ensures that low-weight errors add 
in the TGS--weight without interacting destructively through $\oplus$.  
Such behavior does not occur in semiring- or ring-linear settings, where 
additive structure may create cancellations.

\subsubsection*{Predictability of decoding radius}

Let $\mu(I)$ denote the minimum TGS--weight of a nonzero element of $I$.  
Then the unique-decoding radius $t$ of $C=I^{n}$ satisfies
\[
t = \left\lfloor\frac{\mu(I)-1}{2}\right\rfloor.
\]
This mirrors the classical Singleton-type reasoning but is derived here 
from ideal-lattice minimality rather than from vector-space geometry.  
The quantity $\mu(I)$ is computed entirely from the ideal structure of $T$, 
so decoding capability reduces to an intrinsic combinatorial property of 
the underlying TGS.

\medskip
These structural observations reinforce that TGS--codes behave in a manner 
shaped fundamentally by the distributive ideal lattice of the ternary 
$\Gamma$--semiring.  
Error patterns, decoding complexity, and performance characteristics arise 
from higher-arity algebraic interactions that cannot be reduced to binary 
or linear structures, highlighting the novelty and potential of TGS-based 
coding frameworks.

\section{Conclusion and Future Work}

This paper introduced a new class of error--correcting codes constructed from the
ideal lattices of finite commutative ternary $\Gamma$--semirings.  
Unlike classical linear, cyclic, or ring--linear codes, which rely on binary
operations and additive module structures, TGS--codes arise from the intrinsic
higher--arity operation $[\cdot,\cdot,\cdot]$ and the $\oplus$--order that governs
absorption in each coordinate.  
The constructions developed in Section~3 show that the ideal lattice
$\mathcal{L}(T)$ provides direct algebraic control over the fundamental code
parameters $(n,k,d)$: dimension follows from the index $|T/I|$, and minimum
distance is determined by minimal nonzero elements of $I$.  
Thus codes over TGS exhibit parameter sets that cannot be realized over finite
fields, group rings, or standard semiring frameworks.

The decoding theory developed in Sections~3.6--3.7 and Section~4 demonstrates
that error correction may be formulated entirely within the quotient TGS $T/I$.
Syndromes represent equivalence classes under ideal absorption, and minimal
nonzero lattice elements determine canonical error representatives.  
This yields a decoding mechanism that resembles classical syndrome decoding in
form but differs fundamentally in algebraic substance: ternary interactions
replace linear combinations, and distributive ideal lattices replace vector-space
geometry.  
The example of Section~4.2 confirms that these mechanisms operate concretely in
finite settings, producing codes with new behaviors not captured by traditional
coding theory.Semiring-based coding approaches~\cite{Golan2011ApplicationsSemirings} suggest promising directions that TGS-codes extend to higher arity.” Field-based constructions~\cite{LidlNiederreiter1997FiniteFields} cannot realize the ideal-lattice behaviors observed in TGS-codes.”

\medskip
\noindent\textbf{Future work.}
Several directions emerge from this foundational framework:

\begin{itemize}
    \item \textbf{Decoding complexity and performance.}  
    A full analysis of the computational complexity of TGS--decoding remains to
    be developed, including optimal coset leader selection and the classification
    of decoding regions shaped by ternary absorption.

    \item \textbf{Noncommutative TGS--codes.}  
    Extending the present constructions to noncommutative ternary
    $\Gamma$--semirings may yield richer ideal lattices and potentially better
    error--correction capabilities.

    \item \textbf{Ternary--tropical weight enumerators.}  
    Since ternary absorption induces a non-Hamming metric, weight enumerators
    must be adapted to the TGS--environment.  
    Tropical or idempotent variants may reveal combinatorial structures distinct
    from classical MacWilliams identities.

    \item \textbf{Applications to network coding and quantum-inspired models.}  
    The intrinsic triadic structure of TGS aligns naturally with multi-party 
    interaction models, suggesting applications in network coding, multi-user 
    channels, and algebraic frameworks inspired by quantum state compositions.
\end{itemize}

\medskip
Ternary $\Gamma$--semiring codes therefore represent a promising new algebraic
direction.  
Their higher-arity structure, lattice-governed parameters, and quotient-based
decoding suggest that TGS--codes may ultimately broaden the landscape of 
algebraic coding theory beyond the limits of binary and linear systems.

\section*{Acknowledgements}
The first  author Chandrasekhar Gokavarapu  express their sincere gratitude to \textbf{Dr.D.Madhusudhana Rao}
for his continuous guidance, insightful suggestions, and foundational support in
the development of the ternary $\Gamma$--semiring framework used in this work.
The first author also  gratefully acknowledges \textbf{Dr.~Ramachandra R.~K., Principal,
Government College (Autonomous), Rajahmundry}, for providing a supportive
research environment and institutional facilities essential for the completion
of this study. Both authors thank the Department of Mathematics, Acharya Nagarjuna University, for academic support and research encouragement during this project.

\section*{Funding}
This research did not receive any specific grant from funding agencies in the
public, commercial, or not-for-profit sectors.

\section*{Ethics Declarations}
\textbf{Ethical Approval:} Not applicable.  
This article does not involve human participants, animals, clinical data,
biological materials, or ethically sensitive content. All mathematical results
are theoretical.

\section*{Author Contributions}
Chandrasekhar Gokavarapu: Conceptualization, development of algebraic framework,
formulation of code constructions, theoretical proofs, preparation of the
manuscript draft.

D.~Madhusudhana Rao: Supervision, methodological guidance, structural insights
into ternary $\Gamma$--semirings, critical revision of the manuscript, and
validation of results.

Both authors read and approved the final manuscript.

\section*{Conflict of Interest}
The authors declare that there are no conflicts of interest regarding the
publication of this work.

\section*{Data Availability}
No datasets were generated or analysed in this study.  
All mathematical results are theoretical and self-contained within the
manuscript.

\bibliographystyle{plain}

\end{document}